\documentclass[12pt,reqno]{amsart}

\usepackage{amssymb}
\usepackage{tikz}
\usetikzlibrary{arrows}
\usepackage{enumerate}

\numberwithin{equation}{section}

\theoremstyle{plain}
\newtheorem{theorem}{Theorem}[section]
\newtheorem{proposition}[theorem]{Proposition}
\newtheorem{lemma}[theorem]{Lemma}
\newtheorem{corollary}[theorem]{Corollary}

\theoremstyle{remark}
\newtheorem{remark}[theorem]{Remark}

\newcommand{\QQ}{{\mathbb Q}}
\newcommand{\ZZ}{{\mathbb Z}}

\newcommand{\sym}{{\rm Sym}}

\newcommand{\univ}{{\rm univ}}

\newcommand{\sF}{{\mathcal F}}
\newcommand{\sL}{{\mathcal L}}
\newcommand{\sQ}{{\mathcal Q}}
\newcommand{\sO}{{\mathcal O}}

\begin{document}

\baselineskip=15.5pt

\title[Brauer groups of Quot schemes]{Brauer groups of Quot schemes}

\author[I. Biswas]{Indranil Biswas}

\address{School of Mathematics, Tata Institute of Fundamental Research,
Homi Bhabha Road, Bombay 400005, India}

\email{indranil@math.tifr.res.in}

\author[A. Dhillon]{Ajneet Dhillon}

\address{Department of Mathematics, Middlesex College, University of
Western Ontario, London, ON N6A 5B7, Canada}

\email{adhill3@uwo.ca}

\author[J. Hurtubise]{Jacques Hurtubise}

\address{Department of Mathematics, McGill University, Burnside
Hall, 805 Sherbrooke St. W., Montreal, Que. H3A 0B9, Canada}

\email{jacques.hurtubise@mcgill.ca}



\date{}

\begin{abstract}
Let $X$ be an irreducible smooth complex projective curve. Let 
${\mathcal Q}(r,d)$ be the Quot scheme parametrizing all coherent
subsheaves of ${\mathcal O}^{\oplus r}_X$ of rank $r$ and degree $-d$.
There are natural morphisms ${\mathcal Q}(r,d)\,\longrightarrow\, \text{Sym}^d(X)$
and $\text{Sym}^d(X)\,\longrightarrow\,\text{Pic}^d(X)$. We prove that both
these morphisms induce isomorphism of Brauer groups if $d\, \geq\, 2$.
Consequently, the Brauer group of ${\mathcal Q}(r,d)$ is identified with the
Brauer group of $\text{Pic}^d(X)$ if $d\, \geq\, 2$.
\end{abstract}

\maketitle

\section{Introduction}

Let $X$ be an irreducible smooth projective curve defined over $\mathbb C$.
For any integer $r\, \geq\, 1$, consider the trivial holomorphic vector bundle
${\mathcal O}^{\oplus r}_X$ on $X$. For any $d\, \geq\, 0$,
let ${\mathcal Q}(r,d)$ denote the Quot scheme that parametrizes all torsion
quotients of degree $d$ of the ${\mathcal O}_X$--module ${\mathcal O}^{\oplus r}_X$.
This ${\mathcal Q}(r,d)$ is an irreducible smooth complex projective variety
of dimension $rd$.

For every $Q\, \in\, {\mathcal Q}(r,d)$, we have a corresponding short exact sequence
$$
0\, \longrightarrow\, {\mathcal F}(Q)\,\stackrel{\rho}{\longrightarrow}\,{\mathcal O}^{\oplus r}_X
\, \longrightarrow\, Q \, \longrightarrow\, 0\, .
$$
The pairs $({\mathcal O}^{\oplus r}_X)^*\,=\,
{\mathcal O}^{\oplus r}_X\,\stackrel{\rho^*}{\longrightarrow}\,
{\mathcal F}(Q)^*$ are vortices of a particular numerical type.
The Quot scheme ${\mathcal Q}(r,d)$ is a moduli space of vortices of a
particular numerical type (see \cite{BDW}, \cite{Ba}, \cite{BR} and references
therein).

Sending the above $Q$ to the scheme theoretic support of the quotient for the
homomorphism
$$
\bigwedge\nolimits^r {\mathcal F}(Q) \,\longrightarrow\,
\bigwedge\nolimits^r {\mathcal O}^{\oplus r}_X
$$
induced by the above inclusion ${\mathcal F}(Q)\,\longrightarrow\,
{\mathcal O}^{\oplus r}_X$, we get a morphism
$$
\varphi\, :\, {\mathcal Q}(r,d)\, \longrightarrow\, \text{Sym}^{d}(X)\, .
$$
Sending any $Q\, \in\, {\mathcal Q}(r,d)$ to the holomorphic line bundle
$\bigwedge^r {\mathcal F}(Q)^*$, we get a morphism
$$
\varphi'\, :\, {\mathcal Q}(r,d)\, \longrightarrow\, {\mathcal Q}(1,d)
\,=\, \text{Pic}^d(X)\, .
$$
On the other hand, we have the morphism
$$
\xi_d\, :\, \text{Sym}^{d}(X)\, \longrightarrow\, \text{Pic}^d(X)
$$
that sends any $(x_1\, ,\cdots\, ,x_d)$ to the holomorphic line bundle
${\mathcal O}_X(\sum_{i=1}^d x_i)$. Note that $\varphi'\,=\, \xi_d\circ\varphi$.

The cohomological Brauer group of a smooth complex projective variety $M$ will be
denoted by $\text{Br}'(M)$. A theorem of Gabber says that $\text{Br}'(M)$ coincides
with the Brauer group of $M$ (see \cite{dJ}).

Our aim here is to prove the following:

\begin{theorem}\label{th1}
For the morphisms $\varphi$ and $\xi_d$, the pullback homomorphisms of
Brauer groups
$$
\varphi*\, :\, {\rm Br}'({\rm Sym}^{d}(X))\, \longrightarrow\,
{\rm Br}'({\mathcal Q}(r,d))~\ \text{ and }~\
\xi^*_d\, :\,{\rm Br}'({\rm Pic}^d(X))\, \longrightarrow\,
{\rm Br}'({\rm Sym}^{d}(X))
$$
are isomorphisms provided $d\, \geq\, 2$.
\end{theorem}

Theorem \ref{th1} is proved in Lemma \ref{lem2} and Lemma \ref{lem3}.

If $\text{genus}(X)\,=\, 1$, then ${\rm Sym}^{d}(X)$ is a projective bundle over $X$,
and hence ${\rm Br}'({\rm Sym}^{d}(X))\,=\, 0$. If $\text{genus}(X)\,=\, 0$, then
${\rm Br}'({\rm Sym}^{d}(X))\,=\, 0$ because ${\rm Sym}^{d}(X)\,=\,
{\mathbb C}{\mathbb P}^d$. Therefore, Theorem \ref{th1} has the following
corollary:

\begin{corollary}\label{cor-i}
If ${\rm genus}(X)\,\leq \, 1$, then ${\rm Br}'({\mathcal Q}(r,d))\,=\, 0$.
\end{corollary}

Since ${\mathcal Q}(r,1)$ is a projective bundle over $X$, it follows that
${\rm Br}'({\mathcal Q}(r,d))\, =\, 0$. Note that
${\rm Br}'({\rm Pic}^d(X))$ is nonzero if $\text{genus}(X)\,>\, 1$, while
${\rm Br}'({\rm Sym}^{1}(X))\, =\, 0$.

Fixing a point $x_0\, \in\, X$, construct an embedding
$$
\delta\, :\, {\mathcal Q}(r,d)\, \longrightarrow\, {\mathcal Q}(r,d+r)
$$
by sending any subsheaf ${\mathcal F}\,\subset\,{\mathcal O}^{\oplus r}_X$
to ${\mathcal F}\otimes{\mathcal O}_X(-x_0)$.

The following is proved in Corollary \ref{cor-p}:

\begin{proposition}\label{co-i}
The pullback homomorphism for $\delta$
$$
\delta^*\, :\, {\rm Br}'({\mathcal Q}(r,d+r))\, \longrightarrow\,
{\rm Br}'({\mathcal Q}(r,d))
$$
is an isomorphism if $d\, \geq\, 2$.
\end{proposition}

Now assume that $r,\, \text{genus}(X)\, \geq\, 2$; if $r\,=\, 2$, then also assume that
$\text{genus}(X)\, \geq\, 3$. Iterating the above morphism $\delta$, we get an
ind-scheme. This ind-scheme has the cohomology isomorphic to the moduli stack; see
\cite[Theorem 4.5]{d} or \cite[Chapter 4]{n}. Using
Proposition \ref{co-i} and Theorem \ref{th1}, one can now describe the cohomological
Brauer group of the moduli stack of rank $r$ and degree $d$ bundles.
Further one deduces that the cohomological Brauer
group of the moduli stack of vector bundles on $X$ of rank $r$ and fixed determinant
vanishes. This result was proved earlier in \cite[Theorem 5.2]{BH}. Using this vanishing result it
can be deduced that the cohomological Brauer group of the moduli space of stable vector bundles on
$X$ of rank $r$ and fixed determinant of degree $d$ is a cyclic group
of order $\text{g.c.d.}(r\, ,d)$. This result was proved earlier in \cite{BBGN}.

\section{Cohomological Brauer group}

Let $M$ be an irreducible smooth projective variety defined over $\mathbb C$. Let 
${\mathcal O}^*_M$ denote the multiplicative sheaf on $M$ of holomorphic
functions with values in ${\mathbb C}\setminus \{0\}$. The \textit{cohomological Brauer
group} ${\rm Br}'(M)$ is the torsion subgroup of the
cohomology group $H^2(M,\, {\mathcal O}^*_M)$. 

Let ${\mathcal O}_M$ denote the sheaf of holomorphic functions on $M$.
Consider the short exact sequence of sheaves on $M$
$$
0\, \longrightarrow\, {\mathbb Z}\, \longrightarrow\, {\mathcal O}_M
\, \stackrel{\exp}{\longrightarrow}\,{\mathcal O}^*_M\, \longrightarrow\, 0\, ,
$$
where the homomorphism ${\mathbb Z}\, \longrightarrow\, {\mathcal O}_M$ sends
any integer $n$ to the constant function $2\pi\sqrt{-1}\cdot n$. Let
\begin{equation}\label{e1}
\text{Pic}(M)\,=\, H^1(M,\, {\mathcal O}^*_M)\, \stackrel{c}{\longrightarrow}\,
H^2(M,\, {\mathbb Z})\, \longrightarrow\,H^2(M,\, {\mathcal O}_M) 
\end{equation}
be the corresponding long exact sequence of cohomology groups. The homomorphism
$c$ in \eqref{e1} sends a holomorphic line bundle to its first Chern class. The
image $c(\text{Pic}(M))$ coincides with the N\'eron--Severi group
$$
\text{NS}(M)\, :=\, H^{1,1}(M)\cap H^2(M,\, {\mathbb Z})\, .
$$
Define the subgroup
\begin{equation}\label{e2}
A\, :=\, H^2(M,\, {\mathbb Z})/c(\text{Pic}(M))
\,=\, H^2(M,\, {\mathbb Z})/\text{NS}(M)\, \subset\,
H^2(M,\, {\mathcal O}_M)
\end{equation}
(see \eqref{e1}). Let
$$
H^3(M,\, {\mathbb Z})_{\rm tor}\, \subset\, H^3(M,\, {\mathbb Z})
$$
be the torsion part.

\begin{proposition}[\cite{Sc}]\label{prop1}
There is a natural short exact sequence
$$
0\,\longrightarrow\, A\otimes_{\mathbb Z} ({\mathbb Q}/{\mathbb Z})\,\longrightarrow\,
{\rm Br}'(M)\,\longrightarrow\, H^3(M,\, {\mathbb Z})_{\rm tor} \,\longrightarrow\, 0\, ,
$$
where $A$ is defined in \eqref{e2}.
\end{proposition}

See \cite[p. 878. Proposition 1.1]{Sc} for a proof of Proposition
\ref{prop1}.

\section{The cohomology of symmetric products}\label{scsp}

Let $X$ be an irreducible smooth complex projective curve.
The genus of $X$ will be denoted by $g$.
For any positive integer $d$, let $P_d$
be the group of all permutations of $\{1\, ,\cdots\, ,d\}$. By
$\text{Sym}^d(X)$ we will denoted the quotient of $X^d$ for the
natural action of $P_d$ on it. So $\text{Sym}^d(X)$ parametrizes
all formal sums of the form $\sum_{x\in X}n_x\cdot x$,
where $n_x$ are nonnegative integers with $\sum_{x\in X}n_x\,=\, d$. In other
words, $\text{Sym}^d(X)$ parametrizes all effective divisors on $X$ of degree $d$.
This $\text{Sym}^d(X)$ is an irreducible smooth complex projective variety of
complex dimension $d$. Let
\begin{equation}\label{q}
q_d\, :\, X^d\,\longrightarrow\, \text{Sym}^d(X)\,=\, X^d/P_d
\end{equation}
be the quotient map.

Let $\alpha_1\, , \alpha_2\, , \cdots\, , \alpha_{2g}$ be a symplectic basis for
$H^1(X,\, \ZZ)$ chosen so that
$\alpha_i\cdot \alpha_{i+g}\,=\,1$ for $i\, \leq\, g$, and $\alpha_i\cdot
\alpha_j\,=\,0$ if $\vert i-j\vert\, \not=\, g$. The oriented generator
of $H^2(X,\, \ZZ)$ will be denoted by $\omega$. For $i\, \in\, [1\, ,2g]$
and $j\, \in\, [1\, ,d]$, we have the cohomology classes
\begin{equation}\label{la}
\lambda^j_i \, := \,1\otimes\cdots \otimes \alpha_i \otimes \cdots \otimes 1
\,\in\, H^1(X^n, \, \ZZ)
\end{equation}
and
\begin{equation}\label{et}
 \eta^j \,:=\, 1\otimes \cdots \otimes \omega \otimes \cdots \otimes 1
\,\in\, H^2(X^n,\, \ZZ)\, ,
\end{equation}
where both $\alpha_i$ and $\omega$ are at the $j$-th position.

\begin{theorem}[\cite{Ma}]\label{t:macdonald}
For the morphism $q_d$ in \eqref{q}, the pullback homomorphism
 \[
q^*_d \,: H^*(\sym^d(X),\,\ZZ)\,\longrightarrow\, H^*(X^d,\, \ZZ)
 \]
is injective. Further, the image of $q^*_d$ is generated, as a $\ZZ$--algebra, by
\[
\, \lambda_i \,= \sum_{j=1}^d \lambda_i^j\, ,~ 1\,\le\, i \,\le\, 2g\, ,
\, ~ \ and ~\ \eta = \sum_{j=1}^d \eta^j\, .
\]
\end{theorem}

See \cite[p. 325, (6.3)]{Ma} and \cite[p. 326, (7.1)]{Ma} for Theorem
\ref{t:macdonald}.

There is a universal divisor $D^{\rm univ}$ on $\sym^d(X)\times X$ which consists
of all $(z\, ,x)\,\in\, \sym^d(X)\times X$ such that $x$ is in the support of $z$.
We wish to describe the class of this divisor
in $H^2(\sym^d(X)\times X,\, \ZZ)$. In view of the first part of
Theorem \ref{t:macdonald}, the algebra $H^*(\sym^d(X)\times X,\, \ZZ)$
is considered as a subalgebra of $H^*(X^{d+1},\,\ZZ)$.

For $i\,\in\, [1\, ,d+1]$, let $\pi_i\, :\, X^{d+1}\,\longrightarrow\, X$
be the projection to the $i$--th factor. For any integer
$k\,\in\, [1\, ,d]$, consider the closed immersion
\[
 \iota_k\, :\, X^d\,\hookrightarrow\, X^{d+1}
\]
which is uniquely determined by
\[
 \pi_i \circ \iota_k \,=\,
\begin{cases}
 \pi'_i & \text{if } i\,\ne\, d+1 \\
 \pi'_k & \text{if } i\,=\, d+1
\end{cases}
\]
where $\pi'_j$ is the projection of $X^d$ to the $j$--th factor.
In other words, $i_k(x_1\, , \cdots\, , x_k\, , \cdots\, , x_d)\,=\,
i_k(x_1\, , \cdots\, , x_k\, , \cdots\, , x_d\, , x_k)$.
The divisor on $X^{d+1}$ given by the image of $\iota_k$ will be
denoted by $D_k$.

The above divisor $D_k$ is closely related to the universal divisor
$D^{\rm univ}$ defined above. To see this, consider the projection
$$
q_d\times \text{Id}_X\, :\, X^{d+1} \,=\, X^d\times X^d \,\longrightarrow\,
\text{Sym}^d(X)\times X\, ,
$$
where $q_d$ is constructed in \eqref{q}. The image
$(q_d\times \text{Id}_X)(D_k)$ is independent of the choice of $k$ and it
coincides with $D^{\rm univ}$. This implies that $D^{\rm univ}$ is irreducible.

The classes
\[
 \lambda_i^j \cup \lambda_{i'}^{j'}\, ,\quad i\,\ne\, i'\, ,\ 1\,\le\, j\,<\,j'
\,\le \,d+1\, ,
\]
and 
\[
 \eta^j\, , \quad 1\,\le\, j\,\le\, d+1\, ,
\]
constructed as in \eqref{la} and \eqref{et} for $d+1$,
together give a basis for ${H^2}(X^{d+1},\, \ZZ)$. We have the dual basis for
$H^{2d}(X^{d+1},\, \ZZ)$ given by
\[
 \eta^{j\vee} \,=\,
\omega\otimes\cdots \otimes\omega\otimes 1 \otimes\omega\otimes\cdots \otimes\omega
\]
and
\[
 \left( \lambda_i^j \cup \lambda_{i'}^{j'} \right)^\vee \,=\, \omega\otimes \cdots 
 \otimes\omega\otimes \widetilde{\alpha}_i \otimes\omega\otimes \cdots 
 \omega \otimes \widetilde{\alpha}_{i'}\otimes \omega\otimes\cdots \otimes\omega\, ,
\]
where $\widetilde{\alpha}_i$ (respectively, $\widetilde{\alpha}_{i'}$) is the class with
$\widetilde{\alpha}_{i}\cup\alpha_i \,=\, \omega$ (respectively, $\widetilde{\alpha}_{i'}
\cup\alpha_{i'} \,=\, \omega)$.
Now
\[
 \int_{D_k} \eta^{j\vee} \,=\, \int_{X^d} \iota_k^*\eta^{j\vee} \,=\, 
 \begin{cases}
 1 & j=k \\
 1 & j=d+1 \\
 0 & \text{otherwise}
 \end{cases}
\]
while
\[
 \int_{D_k} \left( \lambda_i^j \cup \lambda_{i'}^{j'} \right)^\vee \,=\,
\int_X \widetilde{\alpha}_i\cup \widetilde{\alpha}_{i'}
\]
if $j'\,=\,d+1$ and $j\,=\, k$, and
$$
\int_{D_k} \left( \lambda_i^j \cup \lambda_{i'}^{j'} \right)^\vee \,=\, 0
$$
otherwise. So the class of $D_k$ is
\[
 \eta^k + \eta^{d+1} + \sum_{i=1}^g \lambda_i^k \cup \lambda_{i+g}^{d+1}
 - \sum_{i=g+1}^{2g} \lambda_i^k \cup \lambda_{i-g}^{d+1}\, .
\]
By the K\"unneth formula, we have
$$
 H^2(\sym^{d}(X)\times X,\,\ZZ)\, \cong\, \left(H^2(\sym^d(X), \,\ZZ)\otimes H^0(X,\,\ZZ) \right)
$$
$$
\oplus \left(H^0(\sym^d(X),\, \ZZ)\otimes H^2(X,\,\ZZ)\right) \oplus
\left(H^1(\sym^d(X),\, \ZZ)\otimes H^1(X,\, \ZZ)\right)\, .
$$
Using Theorem \ref{t:macdonald}
\eqref{t:macdonald}, we have a basis for $H^2(\sym^{d}(X)\times X,\, \ZZ)$
consisting of
\[
 \eta\otimes 1_X\, , \ \{(\lambda_i\cup \lambda_j)\otimes 1_X\}_{i,j=1}^{2g}
\, ,\ 1_{\sym^d(X)}
\otimes \omega\, ,\ \{\lambda_i\otimes \alpha_j\}_{i,j=1}^{2g}\, .
\]
{}From the above computations it follows that the class of $D^{\rm univ}$ is
\begin{equation}\label{D}
 [D^{\rm univ}]\,=\, \eta\otimes 1 + d(1_{\sym^d(X)}\otimes \omega) +
\sum_{i=1}^g \lambda_i\otimes \alpha_{i+g} 
 - \sum_{i=g+1}^{2g} \lambda_i\otimes\alpha_{i-g}\, .
\end{equation}

\begin{proposition}\label{p:classes}
\mbox{} 
\begin{enumerate}
 \item For a fixed point $x_0\,\in\, X$, consider the inclusion 
 \[
 \iota_{x_0} \,:\, \sym^d(X)\hookrightarrow \sym^d(X)\times X
 \]
defined by $z\, \longmapsto\, (z\, ,x_0)$.
The cohomology class $\iota_{x_0}^*[D^{\rm univ}]$ is $\eta$.

\item The slant product of $[D^{\rm univ}]$ with 
$\alpha_i^\vee$ produces the class $\lambda_i$ in $H^1(\sym^d(X),\,\ZZ)$.
 \end{enumerate}
\end{proposition}

\begin{proof}
These follow from \eqref{D}.
\end{proof}

\section{Cohomological Brauer group of the symmetric product}\label{sec3}

Recall that $X$ denotes a smooth projective curve.
Fix a point $x_0\,\in\, X$. For any $d\, \geq\, 1$, let
\begin{equation}\label{e3}
f_d\, :\, \text{Sym}^d(X)\,\longrightarrow\, \text{Sym}^{d+1}(X)
\end{equation}
be the morphism defined by $\sum_{x\in X}n_x\cdot x\, \longmapsto\,
x_0+\sum_{x\in X}n_x\cdot x$. Let
\begin{equation}\label{e4}
f^*_d\, :\, \text{Br}'(\text{Sym}^{d+1}(X))\,\longrightarrow\,
\text{Br}'(\text{Sym}^{d}(X))
\end{equation}
be the pullback homomorphism for $f_d$ in \eqref{e3}.

\begin{lemma}\label{lem1}
For every $d\, \geq\, 2$, the homomorphism $f^*_d$ in \eqref{e4} is
an isomorphism.
\end{lemma}

\begin{proof}
For every positive integer $d$, the cohomology group $H^*(\text{Sym}^d(X),\,
{\mathbb Z})$ is torsionfree by Theorem \ref{t:macdonald}. 
Therefore, from Proposition \ref{prop1} we conclude that
\begin{equation}\label{e5}
{\rm Br}'(\text{Sym}^d(X))\,\cong\,(H^2(\text{Sym}^d(X),\, {\mathbb Z})/
\text{NS}(\text{Sym}^d(X)))\otimes_{\mathbb Z} ({\mathbb Q}/{\mathbb Z})\, .
\end{equation}

{}From Theorem \ref{t:macdonald},
\begin{equation}\label{phi3}
H^2(\text{Sym}^d(X),\, {\mathbb Z})\,=\, (\bigwedge\nolimits^2 H^1(X,\,
{\mathbb Z}))\oplus H^2(X,\, {\mathbb Z})\, .
\end{equation}
Let
$$
f'_d\, :\, H^2(\text{Sym}^{d+1}(X),\,{\mathbb Z})\,\longrightarrow\,
H^2(\text{Sym}^{d}(X),\,{\mathbb Z})
$$
be the homomorphism that sends a cohomology class to its
pullback by the map $f_d$ in \eqref{e3}. It is evident that in terms of
the isomorphism in \eqref{phi3}, this homomorphism $f'_d$ coincides
with the identity map of $(\bigwedge^2 H^1(X,\, {\mathbb Z}))\oplus
H^2(X,\, {\mathbb Z})$.

The isomorphism in \eqref{phi3} is
clearly compatible with the Hodge decompositions. Since $f'_d$ coincides
with the identity map of $(\bigwedge^2 H^1(X,\, {\mathbb Z}))\oplus
H^2(X,\, {\mathbb Z})$, we now conclude that $f'_d$ takes $\text{NS}(
\text{Sym}^{d+1}(X))$ isomorphically to $\text{NS}(
\text{Sym}^{d}(X))$. Therefore, the lemma follows from \eqref{e5}.
\end{proof}

For any positive integer $d$, let
\begin{equation}\label{xi0}
\xi_d\, :\, \text{Sym}^d(X)\,\longrightarrow\, \text{Pic}^d(X)
\end{equation}
be the morphism defined by $\sum_{x\in X}n_x\cdot x\,\longmapsto\,
{\mathcal O}_X(\sum_{x\in X}n_x\cdot x)$. Let
\begin{equation}\label{xi}
\xi^*_d\, :\, \text{Br}'(\text{Pic}^d(X))\,\longrightarrow\,
\text{Br}'(\text{Sym}^d(X))
\end{equation}
be the pullback homomorphism corresponding to $\xi_d$.

\begin{lemma}\label{lem2}
For any $d\, \geq\, 2$, the homomorphism $\xi^*_d$ in \eqref{xi} is an
isomorphism.
\end{lemma}

\begin{proof}
Let $$\eta_d\, :\, \text{Pic}^d(X)\,\longrightarrow\, \text{Pic}^{d+1}(X)$$
be the isomorphism defined by $L\,\longmapsto\, L\otimes{\mathcal O}_X(x_0)$.
We have the commutative diagram
$$
\begin{matrix}
\text{Sym}^d(X) & \stackrel{f_d}{\longrightarrow} & \text{Sym}^{d+1}(X)\\
~\, \Big\downarrow\xi_d && \ ~\Big\downarrow\xi_{d+1}\\
\text{Pic}^d(X)& \stackrel{\eta_d}{\longrightarrow} & \text{Pic}^{d+1}(X)
\end{matrix}
$$
where $f_d$ and $\xi_d$ are constructed in \eqref{e3} and \eqref{xi0} respectively,
and $\eta_d$ is defined above. Let
\begin{equation}\label{B}
\begin{matrix}
\text{Br}'(\text{Pic}^{d+1}(X)) &\stackrel{\eta^*_d}{\longrightarrow} &
\text{Br}'(\text{Pic}^{d}(X))\\
~\, \Big\downarrow\xi^*_{d+1} && \ ~\Big\downarrow\xi^*_d\\
\text{Br}'(\text{Sym}^{d+1}(X)) & \stackrel{f^*_d}{\longrightarrow} &
\text{Br}'(\text{Sym}^{d}(X))
\end{matrix}
\end{equation}
be the corresponding commutative
diagram of homomorphisms of cohomological Brauer groups. From
Lemma \ref{lem1} we know that $f^*_d$ is an isomorphism
for $d\, \geq\, 2$. The homomorphism $\eta^*_d$
is an isomorphism because the map $\eta_d$ is an isomorphism.
Therefore, from the commutativity of \eqref{B} we conclude that the homomorphism
$\xi^*_d$ is an isomorphism if $\xi^*_{d+1}$ is an isomorphism. Consequently, it
suffices to prove the lemma for all $d$ sufficiently large. 

As before, the genus of $X$ is denoted by $g$.
Take any $d\, > \, 2g$. Note that for any line bundle $L$
on $X$ of degree $d$, using Serre duality we have
\begin{equation}\label{f0}
H^1(X,\, L)\,=\, H^0(X,\, K_X\otimes L^\vee)^\vee \,=\, 0
\end{equation}
because $\text{degree}(K_X\otimes L^\vee)\,=\, 2g-2-d \, <\, 0$.

Take a Poincar\'e line bundle ${\mathcal L}\, \longrightarrow\,
X\times \text{Pic}^{d}(X)$. From \eqref{f0} it follows that the direct image
$$
{\rm pr}_*{\mathcal L}\, \longrightarrow\, \text{Pic}^{d}(X)
$$
is locally free of rank $d-g+1$, where
${\rm pr}$ is the natural projection of $X\times \text{Pic}^{d}(X)$
to $\text{Pic}^{d}(X)$. The projective bundle ${\mathbb P}({\rm pr}_*{\mathcal L})$,
that parametrizes the lines in the fibers of the
holomorphic vector bundle ${\rm pr}_*{\mathcal L}$,
is independent of the choice of the Poincar\'e line bundle $\mathcal L$. Indeed,
this follows from the fact that any two choices of the Poincar\'e line bundle
over $X\times \text{Pic}^{d}(X)$ differ
by tensoring with a line bundle pulled back from $\text{Pic}^{d}(X)$
\cite[p. 166]{ACGH}. The total space of ${\mathbb P}({\rm pr}_*{\mathcal L})$ is
identified with $\text{Sym}^{d}(X)$ by sending a section to the divisor
on $X$ given by the section; see \cite{Schwarz}.
This identification between $\text{Sym}^{d}(X)$ and ${\mathbb P}({\rm pr}_*{\mathcal L})$
takes the map $\xi_d$ in \eqref{xi0} to the
natural projection of ${\mathbb P}({\rm pr}_*{\mathcal L})$ to $\text{Pic}^{d}(X)$.

Since ${\mathbb P}({\rm pr}_*{\mathcal L})$ is the projectivization of a vector bundle, the
natural projection
$${\mathbb P}({\rm pr}_*{\mathcal L})\,\longrightarrow\,\text{Pic}^{d}(X)$$
induces an isomorphism of cohomological Bauer groups \cite[p. 193]{Ga}. Consequently,
the homomorphism
$$
\xi^*_d\, :\, \text{Br}'(\text{Pic}^{d}(X))\,\longrightarrow\,
\text{Br}'(\text{Sym}^d(X))
$$
defined in \eqref{xi} is an isomorphism if $d\, > \, 2g$.
We noted earlier that it is enough to prove the lemma for all $d$
sufficiently large. Therefore, the proof of the lemma is now complete.
\end{proof}

\section{The cohomology of the Quot scheme}

For integers $r\, \geq\, 1$ and $d$, denote by
$\sQ(r,d)$ the Quot scheme parametrizing all coherent subsheaves
\[
 \sF\,\hookrightarrow\, {\mathcal O}_X^{\oplus r}
\]
where $\sF$ is of rank $r$ and degree $-d$. Note that there are no such subsheaf
if $d\, <\, 0$. If $d\,=\, 0$, then $\sF\,=\, {\mathcal O}_X^{\oplus r}$.
If $d\,=\, 1$, then $\sQ(r,d)\,=\, X\times {\mathbb C}{\mathbb P}^{r-1}$.
We assume that $d\, \geq\, 1$.

We will now recall from \cite{bifet}
a few facts about the Bia{\l}ynicki-Birula decomposition of $\sQ(r,d)$.
Using the natural action of ${\mathbb G}_m\,=\, {\mathbb C}\setminus\{0\}$
on ${\mathcal O}_X$, we get an action of ${\mathbb G}_m^r$ on
${\mathcal O}_X^{\oplus r}$. This action produces an action of
${\mathbb G}_m^r$ on ${\mathcal Q}(r,d)$. The fixed points of this torus action
correspond to subsheaves of ${\mathcal O}_X^{\oplus r}$
that decompose into compatible direct sums
\[
\bigoplus_{i=1}^r \sL_i\,\hookrightarrow \, {\mathcal O}_X^{\oplus r}\, ,
\]
where $\sL_i\,\hookrightarrow\, {\mathcal O}_X$ is a subsheaf of rank one.
Let $D_i$ be the effective divisor given by the inclusion of $\sL_i$ in
${\mathcal O}_X$. In particular, we have $\sL_i \,=\, {\mathcal O}_X( - D_i)$.

We use the convention that $\text{Sym}^0(X)$ is a single point. Using this notation,
we have
$$
(D_1\, ,\cdots\, ,D_r)\, \in\, {\rm Sym}^{m_1}(X)\times\cdots \times{\rm Sym}^{m_r}(X)\, ,
$$
where $m_i\,=\, \text{degree}(D_i)$. Conversely, if $(D'_1\, ,\cdots\, ,D'_r)\,
\in\, {\rm Sym}^{m_1}(X)\times\cdots \times{\rm Sym}^{m_r}(X)$, then the
point of ${\mathcal Q}(r,d)$ representing the subsheaf
$$
\bigoplus_{i=1}^r {\mathcal O}_X(-D'_i)\, \subset\, {\mathcal O}_X^{\oplus r}
$$
is fixed by the above action of ${\mathbb G}_m^r$ on ${\mathcal Q}(r,d)$.

For $k\, \geq\, 1$,
denote by ${\bf Part}_r^{k}$ the set of partitions of $k$ of length $r$. So
$${\bf m} = (m_1\, ,\cdots \, ,m_r)\,\in \,{\bf Part}_r^k$$ if and only if
$m_j$ are nonnegative integers with
\[
 \sum_{j=1}^r m_j\,=\, k\, .
\]
For ${\bf m}\,\in\, {\bf Part}_r^{d}$, define 
\begin{equation}\label{dm}
d_{\bf m} \,:=\, \sum_{i=1}^r (i-1)m_i\, .
\end{equation}

The connected components of the fixed point locus for the above action of 
${\mathbb G}_m^r$ on $\sQ(r,d)$ are in bijection with the elements of
${\bf Part}_r^{d}$. The
component corresponding to the partition ${\bf m} \,=\, (m_1\, ,\cdots \, ,m_r)$
is the product
\[
\sym^{\bf m}(X)\, :=\,{\rm Sym}^{m_1}(X)\times\cdots \times{\rm Sym}^{m_r}(X)\, .
\]
It is possible, see \cite[page 3]{bifet}, 
 to choose a $1$--parameter subgroup ${\mathbb G}_m\,\longrightarrow\,
{\mathbb G}_m^r$ given by
$z\mapsto (z^{\lambda_1}, z^{\lambda_2},\ldots , z^{\lambda_r})$ so that
the following two hold:
\begin{enumerate}
\item The fixed point locus under the induced action of ${\mathbb G}_m$ is the same
as the fixed point locus under the action of ${\mathbb G}_m^r$. 
\item The integers $\lambda_1< \lambda_2 < \ldots < \lambda_r$ are increasing.
\end{enumerate}

 For
this action of ${\mathbb G}_m$ on $\sQ(r,d)$ define
\[
 \sym^{\bf m}(X)^+ \,:=\, 
 \{x \,\in\, \sQ(r,d) \, \mid\, \lim_{t\to 0}\, t.x \,\in\, \sym^{\bf m}(X) \}\, ,
\]
where ${\bf m}\,\in \,{\bf Part}_r^k$. This stratification of $\sQ(r,d)$
gives us a decomposition of the Poincar\'e polynomial of $\sQ(r,d)$.
Further, the morphism
\begin{equation} \label{eq:fib}
 \sym^{\bf m}(X)^+ \,\longrightarrow\, \sym^{\bf m}(X)
\end{equation}
that sends a point to its limit is a fiber bundle with 
fiber ${\mathbb A}^{d_{\bf m}}$ (see \cite{bb} and \cite{bifet}), where
$d_{\bf m}$ is defined in \eqref{dm}.

This gives
\begin{equation} \label{eq:dim}
 \dim \sym^{\bf m}(X)^+ \,= \,\dim \sym^{\bf m}(X) + d_{\bf m}
\,= \, d + d_{\bf m}
\end{equation}
(see \cite{bifet}).
 
 \begin{theorem}\label{thm1}
For $i\, \geq\, 1$,
\[
H^{i}(\sQ(r,d),\,\ZZ)\,\cong \,\bigoplus_{\substack{{\bf m}\in
{\bf Part}_r^{d} \\ j+2d_{\bf{m}} = i} }
H^{j}({\rm Sym}^{m_1}(X)\times\cdots \times {\rm Sym}^{m_r}(X),\, \ZZ)\, .
\]
\end{theorem}

\begin{proof}
 See \cite{bifet} and \cite[p. 649, Remark]{bifet:94}.
\end{proof}

We will construct some cohomology classes in $H^2(\sQ(r,d),\, \ZZ)$.
There is a universal
vector bundle ${\mathcal F}^{\univ}$ on ${\mathcal Q}(r,d)\times X$.
Fix a point $x_0\, \in\, X$. Let
$$
i_x\, :\, \sQ(r,d)\,\longrightarrow\, \sQ(r,d)\times X
$$
be the embedding defined by $z\, \longmapsto\, (z\, ,x)$.

Let
\begin{equation}\label{x}
c\, :=\, i^*_x c_1({\mathcal F}^{\univ}) \, \in\, H^2({\mathcal Q}(r,d),\,
{\mathbb Z})
\end{equation}
be the pullback. This cohomology class $c$ is clearly independent of $x$.

We can produce cohomology classes
\[
\overline{\alpha}_1\, , \overline{\alpha}_2\, ,\cdots\, ,
\overline{\alpha}_{2g} \,\in\, H^{1}(\sQ(r,d),\,\ZZ)
\]
by taking the slant product of $c_1({\mathcal F}^\univ)$ with the elements of
a basis $\{\alpha_1\, ,\cdots\, , \alpha_{2g}\}$ for
$H^{1}(X,\,\ZZ)$. Finally, there is a cohomology class $\gamma_2\in H^2(\sQ(r,d),\, \ZZ)$ obtained by
taking the slant product of $c_2({\mathcal F}^\univ)$ with the fundamental class of $X$. 

\begin{remark}
{\rm We will see in the next proposition that the cohomology of $\sQ(r,d)$ has
no torsion. The class $c_2({\mathcal F}^\univ)$ is a $(p,p)$-class and
so is the fundamental class of $X$. It follows that the class $\gamma_2$ is
in the N\'eron--Severi subgroup of $\sQ(r,d)$ as the slant product of two $(p,p)$
classes is in fact $(p,p)$.}
\end{remark}

\begin{proposition}\label{prop3}
Suppose that $d\,\geq \,2$. Then the classes 
\[
 c\, ,\ \gamma_2\, ,\ \overline{\alpha}_i\cup \overline{\alpha}_j\, ,\quad
1\,\le\, i\,<\,j\,\le\, 2g\, ,
\]
that generate $H^2(\sQ(r,d),\, \ZZ)$. In fact, 
$H^2(\sQ(r,d),\, \ZZ)$ is torsionfree, and these classes form a basis of
the $\ZZ$--module $H^2(\sQ(r,d),\, \ZZ)$.
\end{proposition}

\begin{proof}
Using Theorem \ref{thm1} and Theorem \ref{t:macdonald} it follows that $H^2(\sQ(r,d),\,
\ZZ)$ is torsionfree of rank 
\[
\binom{2g}{2} + 2\, .
\]
Hence it suffices to show the stated classes generate the second cohomology group.

The torus action on $\sQ(r,d)$ induces a Bia{\l}ynicki-Birula stratification on
this variety, as described above. Using \eqref{eq:dim}, the cell of largest dimension
in the Bia{\l}ynicki-Birula decomposition is the cell corresponding to the partition
\[
{ \bf m}_1 \,=\, (0\, ,0\, ,0\, ,\cdots\, , d)\, ,
\]
and the second largest cell corresponds to the partition
\[
 {\bf m}_2 \,=\, (0\, ,0\, ,\cdots\, , 0\, , 1\, , d-1)\, .
\]
It follows that $\sym^{{\bf m}_1}(X)^+$ is an open dense sub-scheme of $\sQ(r,d)$. 
Let $D \,:=\, \sQ(r,d) \setminus\sym^{{\bf m}_1}(X)^+$ be the complement. 
Using \eqref{eq:fib}, we have
\[
 H^2(\sym^{{\bf m}_1}(X)^+,\, \ZZ)\,=\, H^2(\sym^{{\bf m}_1}(X),\, \ZZ)\, .
\]
Further, by a dimension calculation (\ref{eq:dim}) and a Gysin sequence, 
\[
 H^0(D,\, \ZZ)\,\cong\, H^0(\sym^{{\bf m}_2}(X),\, \ZZ).
\]

Let
\[
\iota\, :\, \sym^{{\bf m}_1}(X) \,\hookrightarrow \,\sQ(r,d)
\]
be the inclusion map.

The Gysin sequence for the decomposition $\sQ(r,d) = \sym^{{\bf m}_1}(X)^+ \coprod D$ now
reads:
\[
\cdots \,\longrightarrow\, H^0(\sym^{{\bf m}_2}(X),\, \ZZ)\,
\stackrel{f_*}{\longrightarrow}\, H^2(\sQ(r,d),\, \ZZ)\,
\stackrel{\iota^*}{\longrightarrow}\,
H^2(\sym^{{\bf m}_1}(X),\, \ZZ)\,\longrightarrow\, \cdots\, ,
\]
where
\begin{equation}\label{fl}
f\, :\, \sym^{d-1}(X)\times X\,\longrightarrow\, \sQ(r,d)
\end{equation}
is the embedding. From \cite{cs}, this sequence splits, or in other
words, the Bia{\l}ynicki-Birula stratification is 
integrally perfect.

To complete the proof, it suffices to verify the following two statements:

\begin{enumerate}[(S1)]
 \item \label{st1} The classes $\iota^*(\overline{\alpha}_i\cup \overline{\alpha}_j)$,
$1\,\le\, i\,<\,j\,\le\, 2g$, and $\iota^*(c)$ generate
 $H^2(\sym^{{\bf m}_1}(X),\, \ZZ)$.
 \item \label{st2} The class $\gamma_2$ generates the image of $f_*$.
\end{enumerate}

For (S\ref{st1}), observe that
$$
\iota^*(\sF^{\rm univ})\,=\, j^*_z{\mathcal O}_{\sym^{d}(X)\times X}(
-D^{\rm univ})\oplus {\mathcal O}_{\sym^{d}(X)}^{r-1}\, ,
$$
where $j_z\, :\, \sym^{d}(X)\, \longrightarrow\, \sym^{d}(X)\times X$
is the embedding defined by $z\, \longmapsto\, (z\, ,x)$, and
$D^{\rm univ}$ is the universal divisor on $\sym^{d}(X)\times X$. From
Proposition \ref{p:classes} and Theorem \ref{t:macdonald} it
follows that the classes
\[
\iota^*(c)\, ,\ \iota^*(\overline{\alpha}_i\cup \overline{\alpha}_j)\, ,
\quad 1\le i<j\le 2g\, , 
\]
give a basis for $H^2(\sym^{{\bf m}_1}(X),\, \ZZ)$. Further, $\gamma_2
\,\in\, {\rm kernel}(\iota^*)$.

For (S\ref{st2}), we assume that $r\,=\,2$ for simplicity. The proof in the case
of higher rank is obtained by adding some trivial summands to the argument below.

As noted above, we have a split exact sequence :
\[
 0\,\longrightarrow\, H^0(\sym^{d-1}(X)\times X,\,\ZZ)\,
\stackrel{f_*}{\longrightarrow}\, H^2(\sQ(r,d),\, \ZZ)\,
 \stackrel{i^*}{\longrightarrow}\, H^2(\sym^d(X),\, \ZZ)\,\longrightarrow\, 0\, .
\]
Fix some quotient
 \[
 q\,:\,{\sO}_X \,\longrightarrow\, Q \,\longrightarrow\, 0
 \]
of degree $d-1$, and also fix some quotient
 \[
q'\,:\,\sO_X \,\longrightarrow\, \sO_p \,\longrightarrow\, 0
 \]
of degree $1$, where $p\,\in\, X$ is a point not in the support of $Q$.
 
This gives us a point $z\,\in\, \sym^{d-1}(X)\times X$. We can expand this
to a morphism
 \[
 F\, :\, {\mathbb A}^1 \longrightarrow (\sym^{d-1}(X)\times X)^+
 \]
by considering the family of quotients
 \[
 F(t) \,:=\, \left(
 \begin{array}{cc}
 qf& 0 \\ tq'& q'
 \end{array}
 \right)
\,:\, {\sO}_X^{\oplus 2} \,\longrightarrow\, Q \oplus \sO_p\,\longrightarrow\, 0\, .
\]

 Taking the closure of $F({\mathbb A}^1)$ in $\sQ(2,d)$ we obtain an inclusion
 \[
F\, :\, {\mathbb P}^1\,\hookrightarrow\, \sQ(2,d)\, .
 \]

As $\dim \sQ(2,d)=2d$, this give a cohomology class
 \[
 [{\mathbb P}^1] \,\in\, H^{4d-2}(\sQ(2,d),\,\ZZ)\, .
 \]
Let
$$
{\mathcal W}\,\longrightarrow\, {\mathbb P}^1\times X
$$
be the restriction of the universal vector bundle ${\mathcal F}^{\univ}
\,\longrightarrow\,\sQ(2,d)\times X$. It fits in the short exact sequence
\begin{equation}\label{shl}
0\,\longrightarrow\, {\mathcal W}\,\longrightarrow\,
{\mathcal O}^{\oplus 2}_{{\mathbb P}^1\times X}\,\longrightarrow\,
\widetilde{Q}\, :=\, ({\mathcal O}_{{\mathbb P}^1}\boxtimes Q)\oplus
{\mathcal O}_{{\mathbb P}^1}(1)\boxtimes {\mathcal O}_p)
\,\longrightarrow\, 0\, .
\end{equation}
Note that the Chern character
\begin{equation}\label{shl2}
Ch(\widetilde{Q})\,=\, d\omega_X +\omega_X\cup \omega_{{\mathbb P}^1}\, ,
\end{equation}
where $\omega_X$ and $\omega_{{\mathbb P}^1}$ are the fundamental classes of
$X$ and ${\mathbb P}^1$ respectively. 
In particular, $c_1(\widetilde{Q})\,=\, \omega_X$. Therefore, the
slant product of $c_1(\widetilde{Q})$ with elements of $H^1(X,\, {\mathbb Z})$
vanish. We have
$$
c_2(\widetilde{Q})\,=\, \omega_X\cup \omega_{{\mathbb P}^1}\, .
$$
Its slant product with $X$ is then just $\omega_{{\mathbb P}^1}$. Therefore,
\[
(F^* \gamma_2) \cup [{\mathbb P}^1] \,= \,\int_{{\mathbb P}^1} \gamma_2 \,= \,1\, .
\]
 So the cohomology classes described in the statement of the proposition
 give a basis for the vector space $H^2(\sQ(r,d), \, \QQ)$.

We will prove the following statements:
\begin{equation}\label{C1}
 (F^* c) \cup [{\mathbb P}^1] \,=\, 0\, .
\end{equation}
\begin{equation}\label{C2}
 \alpha_i \cup [{\mathbb P}^1] \,=\, 0\, .
\end{equation}
\begin{equation}\label{C3}
f_*([\sym^{d-1}(X)\times X])\cup [{\mathbb P}^1] \,=\, 1\, .
\end{equation}
\begin{equation}\label{C4}
 (F^*\gamma_2) \cup [{\mathbb P}^1] \,= \,\int_{{\mathbb P}^1} \gamma_2 \,=\, 1\, .
\end{equation}
The map $f$ is defined in \eqref{fl}.

We first show that the above statements complete the proof. For that it
is sufficient to observe that they imply both
 \[
 f_*([\sym^{d-1}(X)\times X])\quad\text{ and }\quad \gamma_2
 \]
are dual to $[{\mathbb P}^1]$ and hence must be equal.

To prove \eqref{C1}, consider \eqref{shl}.
Choose a point $x\,\in\, X$ away from the support of $Q\oplus \sO_p$
and restrict $\mathcal W$ to ${\mathbb P}^1\times \{x\}$. From \eqref{shl2}
it follows that the first Chern class of this restriction vanish. The
first Chern class of this restriction clearly coincides with
$(F^* c)\cup [{\mathbb P}^1]$.

The left--hand side of \eqref{C2} is clearly the slant product of
$c_1(\mathcal W)$ with $\alpha_i$. We noted above that this slant product vanishes.

Now, \eqref{C3} is clear from the construction of the morphism $F$ from
${\mathbb P}^1$. Finally \eqref{C4} has already been proved.
\end{proof}

\section{The cohomological Brauer group}

For integers $r\, \geq\, 1$ and $d\, \geq\, 0$, take any subsheaf
${\mathcal F}\,\subset\, {\mathcal O}^{\oplus r}_X$ lying in
${\mathcal Q}(r,d)$. Taking $r$-th exterior power, we get a subsheaf
$\bigwedge^r{\mathcal F}\,\subset\, \bigwedge^r{\mathcal O}^{\oplus r}_X
\,=\, {\mathcal O}_X$. Let
\begin{equation}\label{v1}
\varphi\, :\, {\mathcal Q}(r,d)\, \longrightarrow\, \text{Sym}^{d}(X)
\end{equation}
be the morphism that sends any subsheaf ${\mathcal F}\,\subset\,
{\mathcal O}^{\oplus r}_X$ to the scheme theoretic support of the above quotient
${\mathcal O}_X/\bigwedge^r{\mathcal F}$. Let
\begin{equation}\label{v}
\varphi^*\, :\, \text{Br}'(\text{Sym}^{d}(X))\, \longrightarrow\,
\text{Br}'({\mathcal Q}(r,d))
\end{equation}
be the pullback homomorphism using $\varphi$.

\begin{lemma}\label{lem3}
The homomorphism $\varphi^*$ in \eqref{v} is an isomorphism.
\end{lemma}

\begin{proof}
Note that $\text{Br}'({\mathcal Q}(r,d))\,=\, \text{Br}'(\text{Sym}^{d}(X))
\,=\,0$ if $d\,\leq\, 1$. Therefore, we assume that $d\, \geq\, 2$.

The cohomology group $H^3({\mathcal Q}(r,d),\, {\mathbb Z})$ is torsionfree.
Indeed, this follows from Theorem \ref{thm1} and the fact that $H^*(\text{Sym}^n(X),
\, {\mathbb Z})$ is torsionfree \cite[p. 329, (12.3)]{Ma}. Therefore, Proposition
\ref{prop1} says that
\begin{equation}\label{brq}
\text{Br}'({\mathcal Q}(r,d))\,=\, (H^2({\mathcal Q}(r,d),\, {\mathbb Z})/
\text{NS}({\mathcal Q}(r,d)))\otimes_{\mathbb Z}({\mathbb Q}/{\mathbb Z})\, .
\end{equation}

Let
$$
\varphi'\, :\, H^2(\text{Sym}^{d}(X),\, {\mathbb Z})\,\longrightarrow\,
H^2({\mathcal Q}(r,d),\, {\mathbb Z})
$$
be the pullback homomorphism using $\varphi$ in \eqref{v1}. Recall from
Theorem \ref{t:macdonald} the description of $H^2(\text{Sym}^{d}(X),\,
{\mathbb Z})$. From Proposition
\ref{prop3} we conclude that $\varphi'$ is injective, and
\begin{equation}\label{tc}
H^2({\mathcal Q}(r,d),\, {\mathbb Z})\,=\, \text{image}(\varphi')\oplus
{\mathbb Z}\cdot \gamma_2\, ,
\end{equation}
where $\gamma_2$ is the cohomology class in Proposition \ref{prop3}.
Take any point $$y\,:=\, (y_1\, ,\cdots\, ,y_{d})\, \in\, \text{Sym}^{d}(X)$$ such
that all $y_i$ are distinct. Then $\varphi^{-1}(y)$ is a product of copies
of ${\mathbb C}{\mathbb P}^{r-1}$, hence $$H^1(\varphi^{-1}(y),\, {\mathbb Z})
\,=\, 0\, .$$ From this it follows that the image of the cup product
$$
H^1({\mathcal Q}(r,d),\, {\mathbb Z})\otimes H^1({\mathcal Q}(r,d),\, {\mathbb Z})
\,\longrightarrow\,H^2({\mathcal Q}(r,d),\, {\mathbb Z})
$$
is in the image $\varphi'$. If the point $x\, \in\, X$ in \eqref{x} is different from
all $y_i$, then the restriction of the universal vector bundle ${\mathcal F}^{\univ}$
(see \eqref{x}) to $\varphi^{-1}(y)$ is the trivial vector bundle of rank $r$. From
this it follows that $c$ is in the image of $\varphi'$.

{}From \eqref{tc} it follows immediately that
$$
\text{NS}({\mathcal Q}(r,d))\,=\, \varphi'(\text{NS}(\text{Sym}^{d}(X)))
\oplus {\mathbb Z}\cdot \gamma_2\, .
$$
In view of \eqref{brq}, from this we conclude that
$\varphi^*$ in \eqref{v} is an isomorphism if $d\, \geq\, 2$.
\end{proof}

As before, fix a point $x_0\, \in\, X$. Let
\begin{equation}\label{dede}
\delta\, :\, {\mathcal Q}(r,d)\, \longrightarrow\, {\mathcal Q}(r,d+r)
\end{equation}
be the morphism that sends any ${\mathcal F}\, \subset\, {\mathcal O}^{\oplus r}_X$
represented by a point of ${\mathcal Q}(r,d)$ to the point representing
the subsheaf ${\mathcal F}\otimes{\mathcal O}_X(-x_0)\, \subset\, {\mathcal
O}^{\oplus r}_X$. Let
\begin{equation}\label{delta}
\delta^*\,:\, \text{Br}'({\mathcal Q}(r,d+r))\, \longrightarrow\,
\text{Br}'({\mathcal Q}(r,d))
\end{equation}
be the pullback homomorphism by $\delta$.

\begin{corollary}\label{cor-p}
For any $d\, \geq\, 2$, the homomorphism $\delta^*$ in \eqref{delta}
is an isomorphism.
\end{corollary}

\begin{proof}
As in \eqref{v1}, define
$$
\psi\, :\, {\mathcal Q}(r,d+r)\, \longrightarrow\, \text{Sym}^{d+r}(X)
$$
to be the morphism that sends any subsheaf ${\mathcal F}\,\subset\,
{\mathcal O}^{\oplus r}_X$ to the scheme theoretic support of the corresponding quotient
$(\bigwedge^r{\mathcal O}^{\oplus r}_X)/(\bigwedge^r{\mathcal F})$. Let
$$
h\, :\, \text{Sym}^d(X)\,\longrightarrow\, \text{Sym}^{d+r}(X)
$$
be the morphism defined by $\sum_{x\in X}n_x\cdot x\, \longmapsto\,
r\cdot x_0+\sum_{x\in X}n_x\cdot x$. The following diagram of morphisms
$$
\begin{matrix}
{\mathcal Q}(r,d) & \stackrel{\delta}{\longrightarrow} & {\mathcal Q}(r,d+r)\\
~\Big\downarrow \varphi &&~\Big\downarrow \psi\\
\text{Sym}^d(X) & \stackrel{h}{\longrightarrow} & \text{Sym}^{d+r}(X)
\end{matrix}
$$
is commutative, where $\varphi$ and $\delta$ are defined in \eqref{v1} and 
\eqref{dede} respectively. Consider the corresponding commutative diagram
$$
\begin{matrix}
\text{Br}'(\text{Sym}^{d+r}(X)) &\stackrel{h^*}{\longrightarrow}&
\text{Br}'(\text{Sym}^{d}(X))\\
~\Big\downarrow \psi^* &&~\Big\downarrow \varphi^* \\
\text{Br}'({\mathcal Q}(r,d+r)) & \stackrel{\delta^*}{\longrightarrow} &
\text{Br}'({\mathcal Q}(r,d))
\end{matrix}
$$
of homomorphisms.
If $d\, \geq\, 2$, from Lemma \ref{lem3} we know that $\psi^*$ and
$\varphi^*$ are isomorphisms, while Lemma \ref{lem1} implies that
$h^*$ is an isomorphism. Therefore, the homomorphism $\delta^*$ is an isomorphism.
\end{proof}

\begin{remark}\label{rem-r}
We are grateful to an unknown referee for this comment. We give here an
alternative proof of the fact that the pullback map $\varphi^*$ induces an 
isomorphism on cohomology. Consider 
the big cell of the Bia{\l}ynicki-Birula decomposition described above. It
corresponds to the partition 
\[
{ \bf m}_1 \,=\, (0\, ,0\, ,0\, ,\cdots\, , d).
\]
We have a Zariski locally trivial fibration 
\[
\rho\,:\, \sym^{{\bf m}_1}(X)^+ \,\longrightarrow\, \sym^{{\bf m}_1}(X)
\]
with fiber ${\mathbb A}^n$, see \cite[p. 492]{BB}.
We claim that we have an induced isomorphism 
\[
{\rm Br}'(\sym^{{\bf m}_1}(X)^+) \cong {\rm Br}'(\sym^{{\bf m}_1}(X)).
\]
To see this we will make use of the exact sequence in Proposition \ref{prop1} which 
is valid for non-compact spaces, see \cite[p. 878]{Sc}. 
The morphism $\rho$ induces an isomorphism in cohomology groups as
it has contractible fibers.
Although the morphism $\rho$ may not be a vector bundle, the N\'{e}ron-Severi
groups of the two varieties agree under the identification of
cohomology groups above, see \cite[p. 22, Proposition 1.9]{fulton}. It follows now
from Proposition \ref{prop1} and the 5 lemma that $\rho^*$ induces an isomorphism 
on cohomological Brauer groups.

The morphism $\varphi\,:\,\sym^{{\bf m}_1}(X)\,\longrightarrow\, \sym^d(X)$ is an isomorphism.
So we have a diagram 
\begin{center}
\begin{tikzpicture}
\node (TL) at (0,0) {${\rm Br}'(\sym^{{\bf m}_1}(X)^+)$};
\node (TR) at (4,0) {${\rm Br}'({\mathcal Q}(r,d))$};
\node (B) at (0,-2) {${\rm Br}'( \sym^{d}(X))$,};
\draw [left hook->] (TR) edge node [above] {$\iota^*$} (TL) ; 
\draw [->] (B) edge node [left] {$\varphi^*$} (TL);
\draw [->] (B) edge node [right] {$\varphi^*$} (TR);
\end{tikzpicture}
\end{center}
$\iota$ is the composition $\sym^{{\bf m}_1}(X)^+\,\hookrightarrow\,
\sym^{{\bf m}_1}(X)\,\longrightarrow\, {\mathcal Q}(r,d)$;
we note that the homomorphism $\varphi^*$ in the left is an isomorphism.
The map $i^*$ is injective by \cite[IV, Corollary 2.6]{milne}. We can
now deduce that $\varphi^*$ is an isomorphism.
\end{remark}

\section*{Acknowledgements}

We thank the two referees for detailed and helpful comments.
Remark \ref{rem-r} is due to a referee. The first-named
author is supported by a J. C. Bose Fellowship.

\end{document}